\documentclass[reqno]{amsart}

\usepackage{cite}
\usepackage{graphicx,subfigure,xcolor}

\numberwithin{equation}{section}

\usepackage[latin1]{inputenc}
\usepackage[english]{babel}

\usepackage{amsmath,amsthm,amsfonts,latexsym,amssymb}
\usepackage[colorlinks]{hyperref}
\hypersetup{linkcolor=blue,citecolor=blue,filecolor=black,urlcolor=blue}
\usepackage{comment}

\usepackage{color}

{ \theoremstyle{plain}
\newtheorem{theorem}{Theorem}[section]
\newtheorem{proposition}[theorem]{Proposition}
\newtheorem{lemma}[theorem]{Lemma}
\newtheorem{corollary}[theorem]{Corollary}
  \theoremstyle{remark}
\newtheorem{remark}[theorem]{Remark}
  \theoremstyle{definition}

}

\def\R{\mathbb{R}}

\def\N{\mathbb{N}}

\begin{document}


\title[]{A supercritical elliptic equation in the annulus}

\author[A. Boscaggin]{Alberto Boscaggin}
\address{Alberto Boscaggin\newline\indent
Dipartimento di Matematica
\newline\indent
Università di Torino
\newline\indent
via Carlo Alberto 10, 10123 Torino, Italy}
\email{alberto.boscaggin@unito.it}

\author[F. Colasuonno]{Francesca Colasuonno}
\address{Francesca Colasuonno\newline\indent
Dipartimento di Matematica
\newline\indent
Università di Bologna
\newline\indent
p.zza di Porta San Donato 5, 40126 Bologna, Italy}
\email{francesca.colasuonno@unibo.it}

\author[B. Noris]{Benedetta Noris}
\address{Benedetta Noris \newline \indent 
Dipartimento di Matematica \newline\indent
Politecnico di Milano \newline\indent
p.zza Leonardo da Vinci 32, 20133 Milano, Italy}
\email{benedetta.noris@polimi.it}

\author[T. Weth]{Tobias Weth}
\address{Tobias Weth\newline \indent
Institut f\"ur Mathematik\newline \indent
Goethe-Universit\"at Frankfurt\newline \indent
Robert-Mayer-Str. 10, D-60629 Frankfurt am Main, Germany}
\email{weth@math.uni-frankfurt.de}


\begin{abstract} 
By a combination of variational and topological techniques in the presence of invariant cones, we detect a new type of positive axially symmetric solutions of the Dirichlet problem for the elliptic equation
$$
-\Delta u + u = a(x)|u|^{p-2}u
$$
in an annulus $A \subset \R^N$ ($N\ge3$). Here $p>2$ is allowed to be supercritical and $a(x)$ is an axially symmetric but possibly nonradial function with additional symmetry and monotonicity properties, which are shared by the solution $u$ we construct. 
In the case where $a$ equals a positive constant, we detect conditions,  only depending on the exponent $p$ and on the inner radius of the annulus, that ensure that the solution is nonradial.
\end{abstract}

\maketitle

\section{Introduction}

In the present paper we are concerned with the nonlinear elliptic equation
\begin{equation}
  \label{eq:supercritical-general}
-\Delta u + u = a(x)|u|^{p-2}u
\end{equation}
in a subset of $\R^N$,
in the case where $N \ge 3$,  $x \mapsto a(x)$ is a positive weight function,  and the nonlinearity is (possibly) supercritical, i.e., $p > 2^*:=2N/(N-2)$. 
In the supercritical regime, a major obstruction to the search of solutions of (\ref{eq:supercritical-general}) is the lack of embeddings of the Sobolev space $H^1(\R^N)$ into the integrability space $L^p(\R^N)$.  As a consequence,  the equation (\ref{eq:supercritical-general}) does not admit a variational framework in $H^1(\R^N)$.
The same is true for the Dirichlet and Neumann problem for (\ref{eq:supercritical-general}) in a bounded domain $\Omega \subset \R^N$, as neither $H^1(\Omega)$ nor $H^1_0(\Omega)$ is embedded in $L^p(\Omega)$ if $p>2^*$. Hence standard variational methods do not apply in these cases
and, more in general, compactness issues have to be faced.
Incidentally, let us recall that, due to the Pohozaev identity, the Dirichlet problem for (\ref{eq:supercritical-general}) does not admit nontrivial solutions in a bounded star-shaped domain $\Omega$ if $p\ge 2^*$ and $a$ is positive constant weight function, see e.g. \cite{Willem}. On the other hand, while no obstruction for the solvability of the Dirichlet problem for (\ref{eq:supercritical-general}) is known in the case of topologically nontrivial domains, 
few results in the literature deal with this problem in the full supercritical regime. 

In the present paper we wish to show that the combination of variational and topological methods in the spaces $H^1_0(\Omega)$ and $C^1_0(\Omega)$ can yield existence of positive solutions of the Dirichlet problem for (\ref{eq:supercritical-general}), in the case where a cone of functions with suitable invariance properties can be found. The presence of invariant cones, characterized by monotonicity properties of functions, has already been exploited in \cite{BNW,ST,CM} to construct solutions of the Neumann problem for the supercritical equation (\ref{eq:supercritical-general}) in specific domains, see also \cite{ABF-ESAIM,ABF-PRE,CN} for related results. More precisely, in \cite{BNW,ST,ABF-ESAIM,ABF-PRE,CN} the case of a ball $\Omega \subset \R^N$ and a radial and radially increasing function $a$ is considered, while \cite{CM} is devoted to domains given as a product of lower-dimensional balls and a function $a$ with associated symmetry and monotonicity properties. A key difference between these papers dealing with the Neumann problem and our present work is that, by construction, the solutions found in
the cited articles attain their maximum on the boundary of the underlying domain, which cannot be realized for the corresponding Dirichlet problem.  
The fact that we work with solutions that are not radial introduces substantial additional difficulties to the problem under consideration. Our strategy to overcome these obstacles is to combine some of the variational techniques used in \cite{BNW,CN} with a non-variational approach inspired from \cite{B2001}.  We consider this fruitful interaction between different tools of the analysis to be one of the most relevant  achievements of the present work and we expect that it may lead in the future to different applications.
For further existence results for supercritical Neumann problems we refer to \cite{dPMP,dPPV}.  

In this paper we focus on the problem 
\begin{equation}\label{P}
  \left\{
  \begin{aligned}
  -\Delta u + u &= a(x)u^{p-1}
&&\qquad \mbox{in }A,\\
u&>0 &&\qquad \mbox{ in }A,\\
u&=0 &&\qquad \mbox{on }\partial A,
  \end{aligned}
\right.
\end{equation}
where $A$ is a bounded $N$-dimensional annulus $A :=\{x\in\mathbb{R}^N\,:\, R_0<|x|<R_1\}$ with $N\ge 3$ (here, $0<R_0 <R_1<\infty$). As it is well known, the existence of a radial solution to \eqref{P} can be easily proved, for any $p > 2$, assuming $a$ is a radial positive bounded function. Our aim is to investigate problem \eqref{P} for a certain class of possibly nonradial but axially symmetric weight functions $a$. We point out that the restriction to axially symmetric functions alone does not help to overcome the lack of a variational structure and compactness properties in the supercritical case, since axially symmetric functions may concentrate on the symmetry axis which has a nonempty intersection with the annulus $A$.  
Existence and multiplicity results for problems similar to (\ref{P}) have been obtained in \cite{gladiali-et-al2010} by means of bifurcation techniques,  being the annulus fixed and the exponent $p$ the bifurcation parameter,  and in \cite{ByeonKimPistoia2013} relying on the Lyapunov--Schmidt reduction argument, in the case of expanding annuli with fixed width.
In the very recent preprint \cite{CowanMoameni2021}, the authors impose the same monotonicity properties that we have; they take advantage of the invariant cone by applying a convex analysis approach and by working in the dual space.
We also wish to mention that relevant existence results for specific related critical and slightly supercritical Dirichlet problems can be found in \cite{BC,dPW,P,P2}. 

In order to state our result precisely,  let us introduce some notation. Assuming without loss of generality that the axis of symmetry is the $x_N$-axis, we call a function on $\overline{A}$ axially symmetric if it only depends on
\[
r=|x|\in [R_0,R_1] \qquad \text{and}\qquad \theta=\arcsin\left(\frac{x_N}{r}\right)\in \left[-\frac{\pi}{2},\frac{\pi}{2}\right].
\]
Hence every axially symmetric function $u$ on $\overline{A}$ can be written as 
$$
u(x) = \mathfrak{u}\left(|x|,\arcsin\left(x_N/r\right)\right)\qquad \text{with a function $\mathfrak{u}: [R_0,R_1]\times\left[-\frac{\pi}{2},\frac{\pi}{2}\right] \to \mathbb{R}$.}
$$
To describe further related symmetry and monotonocity properties, we introduce the cone
\begin{equation}\label{cone}
\widehat{\mathcal K}:= \left\{u\in C^1(\overline A)\,:\, \begin{aligned} & u=\mathfrak{u}(r,\theta), \,u\ge 0 \mbox{ in } \overline{A},\\
&\mathfrak{u}(r,\theta)=\mathfrak{u}(r,-\theta) \mbox{ in }[R_0,R_1]\times(0,\pi/2) ,\\ 
& \mathfrak{u}_\theta(r,\theta)\le 0 \mbox{ in }[R_0,R_1]\times(0,\pi/2). 
\end{aligned}
\right\},
\end{equation}
where $\mathfrak{u}_\theta$ stands for the partial derivative with respect to the variable $\theta$.
Notice that a function $u\in\widehat{\mathcal K}$ satisfies also $\mathfrak{u}_\theta(r,\theta)\ge 0 \mbox{ in } [R_0,R_1]\times(-\pi/2,0).$ We also set
\begin{equation}\label{eq:K_tilde_def}
  \mathcal K := \{u \in \widehat{\mathcal{K}} \::\: u \big|_{\partial A} \equiv 0 \}.
\end{equation}
Hence $\mathcal K$ is the intersection of $\widehat{\mathcal{K}}$ with the function space 
\begin{equation}\label{eq:C10A_def}
C^1_0(A):= \{u \in C^1(\overline A)\::\: u\big|_{\partial A}\equiv 0\} \subset H^1_0(A),
\end{equation}
which will play a central role in the variational approach we propose in this paper. With this notation, we assume that
\begin{equation}\label{a_assumptions}
a\in \widehat{\mathcal{K}}, \qquad a> 0 \quad \text{in } \overline{A},
\end{equation}
and we will show that \eqref{P} admits a nontrivial solution belonging to $\mathcal K$ and enjoying a suitable minimality property. To describe it precisely, we define the functional 
\begin{equation}\label{eq:I_def}
I:C^1_0(A)\to\mathbb R, \qquad I(u):=\frac{1}{2}\int_A(|\nabla u|^2+u^2)dx-\frac{1}{p}\int_A a(x)|u|^p dx,
\end{equation}
which is well known to be well-defined and of class $C^2$, since $p > 2$.
Within the cone $\mathcal K$, we also consider the following Nehari-type set
\[
\mathcal N_\mathcal K:= \{u \in \mathcal{K}\::\: u \not\equiv 0, \,I'(u)u=0\}
\]
and the Nehari value
\begin{equation}\label{eq:Ic-equivalent}
c_I := \inf_{u \in {\mathcal N}_{\mathcal{K}}} I(u).
\end{equation}

With this notation, our first main result now reads as follows.

\begin{theorem}\label{thm:main_existence}
  Let $N\geq3$, $p>2$, and suppose that the function $a$ satisfies (\ref{a_assumptions}). Then we have $c_I >0$, and $c_I$ is attained in ${\mathcal N}_{\mathcal{K}}$. Moreover, every minimizer $u \in \mathcal N_\mathcal{K}$ of $I\big|_{\mathcal N_\mathcal{K}}$ is a nontrivial solution of \eqref{P} belonging to $\mathcal{K}$.
\end{theorem}
Hereafter we will call every minimizer $u \in \mathcal N_\mathcal{K}$ of $I\big|_{\mathcal N_\mathcal{K}}$ a {\em $\mathcal{K}$-ground state solution of (\ref{P})}.

In the case where $a$ is a nonradial function, every solution of (\ref{P}) is nonradial, and Theorem~\ref{thm:main_existence} yields a new existence result for such solutions in the case of critical or supercritical exponents $p>2$.

Let us say a few words about the proof of Theorem \ref{thm:main_existence}. 
The starting point of our method is inspired by the papers \cite{BNW,ST}, 
that concern radial solutions to the Neumann problem for equation (\ref{eq:supercritical-general}) in a ball (see also \cite{ABF-ESAIM,ABF-PRE,CN,CM}). 
As shown in \cite{BNW,ST},  in the case of Neumann boundary conditions, restricting the attention to radial and radially increasing functions provides a priori bounds that are sufficient to overcome the lack of a global variational structure and the lack of compactness.  As a consequence, in the cited papers, one can prove the existence of solutions by restricting variational arguments to the cone of positive radial and radially increasing functions.
In a similar spirit, we will exploit here the fact that the functions in the cone $\mathcal K$ enjoy some a priori bounds.  We stress that the present nonradial framework requires major modifications of the approach: not only the a priori bounds are more difficult to obtain,  but also, being weaker, we are not allowed to proceed by restricting variational arguments to the cone. At this point, what we consider the most interesting part of the paper comes into play: we combine the mountain pass theorem restricted to the cone with a non-variational approach inspired from \cite{B2001}. 

Let us explain our argument in more detail.  We first notice that the minimum value $c_I$ can be reinterpreted as a minimax value of mountain pass-type, cf. Lemma \ref{le:cI=dI}.
Hence, we develop a mountain pass-type argument for the functional $I$ in the cone $\mathcal{K}$, by replacing the usual gradient flow with 
$$
\frac{d}{dt} \eta(t,u) = - (\textnormal{Id} - T)(\eta(t,u)), 
$$
where $T: C^1_0(A) \to C^1_0(A)$ is the operator defined by 
\[
T(u) := (-\Delta + \textnormal{Id})^{-1}(a(x)|u|^{p-2}u).
\]
Notice that, with this approach, solutions will be provided as fixed points of the operator $T$.
A major difficulty lies in the fact that the functional $I$ is of class $C^2$ in the space $C^1_0(A)$, whereas the compactness properties are available with respect to the $H^1(A)$-norm
(see the Palais-Smale type condition proved in Lemma \ref{lemma:PS}).
In order to overcome this obstacle, we adopt a dynamical system point of view, partially inspired from \cite{B2001}. More precisely, via the descent flow,  we manage to construct a sequence belonging to the boundary of a certain domain of attraction, that converges to a fixed point of $T$, that is a solution of the problem.

\medskip

 In the following, we wish to discuss the important special case where $a$ is a constant function. In this case we may, by renormalization, assume that $a \equiv 1$. Then (\ref{P}) reduces to the problem
\begin{equation}\label{P-const}
\begin{cases}
-\Delta u + u = |u|^{p-2}u\quad&\mbox{in }A,\\
u>0 &\mbox{ in }A,\\
u=0&\mbox{on }\partial A
\end{cases}
\end{equation}
which has received widespread attention with regard to the existence and shape of radial and nonradial solutions. 
As already observed, a radial solution exists for any $p > 2$ and, by \cite{T}, it is unique. 
We recall that, when the domain is a ball, the celebrated symmetry result by Gidas-Ni-Nirenberg \cite{GNN} ensures that every positive regular solution is radial. This symmerty preservation still holds when the domain $A$ is an annulus with a small hole and the nonlinearity is subcritical, see \cite{GPY}. On the other hand, for expanding annuli with fixed difference of radii, the existence of multiple nonradial solutions has been proved in \cite{Coffman} in the two-dimensional case,  in \cite{B,CW} for $N\ge 3$ in the subcritical regime 
and in \cite{ByeonKimPistoia2013} by means of Lyapunov--Schmidt reduction. 
For some supercritical nonlinearities, yet subject to certain growth conditions at infinity, the existence of nonradial solutions has been obtained in \cite{Li}. Other existence results for nonradial solutions of \eqref{P-const} in the supercritical case are obtained in \cite{gladiali-et-al2010} via bifurcation techniques.

When applying Theorem~\ref{thm:main_existence} to equation \eqref{P-const}, a priori it is not clear whether the solution found is radial or not.
We detect a condition sufficient to ensure that the $\mathcal{K}$-ground state solution of \eqref{P} is not radial, namely
\begin{equation}\label{eq:suff_p}
p \ge 2+ \frac{2N}{\bigl(\frac{N-2}{2}\bigr)^2+ R_0^2},
\end{equation}
$R_0$ being the inner radius of the annulus. More precisely,  we prove the following result concerning problem \eqref{P-const}.

\begin{theorem}\label{thm:main-a-const}
Let $N\geq3$ and relation \eqref{eq:suff_p} hold.  
Then every $\mathcal{K}$-ground state solution of (\ref{P-const}) is nonradial.
\end{theorem}  

 Remarkably,  condition \eqref{eq:suff_p} only involves $p$ and $R_0$, meaning that our existence result holds for any outer radius $R_1 > R_0$. In particular, if $0<R_0<R_1$ are given, we obtain a nonradial solution of (\ref{P-const}) if $p \ge 2+ \frac{2N}{\bigl(\frac{N-2}{2}\bigr)^2+ R_0^2}$.  As a consequence, we also deduce the following.

\begin{corollary}
\label{cor-any-annulus}  
  Let $N\geq3$ and suppose that $p> 2 + \frac{8N}{(N-2)^2}$. Then, for any annulus $A :=\{x\in\mathbb{R}^N\,:\, R_0<|x|<R_1\}$ with arbitrary $0<R_0<R_1$, every $\mathcal{K}$-ground state solution of (\ref{P-const}) is nonradial. Hence (\ref{P-const}) admits a nonradial solution on any annulus in this case.   
\end{corollary}  

If $p>2$ is given, then \eqref{eq:suff_p} amounts to the explicit condition
  $$
  R_0^2 \ge \frac{2N}{p-2}- \left(\frac{N-2}{2}\right)^2
  $$
  on the inner radius which guarantees the existence of a nonradial solution of (\ref{P-const}).
  In this regard,  we improve the existence results in \cite{ByeonKimPistoia2013,gladiali-et-al2010} both by providing a quantitative relation between $p$ and $R_0$ sufficient for the existence of nonradial solutions, and by proving the existence of nonradial solutions regardless of $R_1$.

Theorem~\ref{thm:main-a-const} also complements a result in \cite{gladiali-et-al2010} on local bifurcation of nonradial solutions for problem (\ref{P-const}). More precisely, it is shown in \cite[Theorem 1.7]{gladiali-et-al2010} that there exists an ordered sequence $\{p_k\}_k$ with $\lim \limits_{k \to \infty}p_k = \infty$ of bifurcation points in the sense that, for every fixed $k \in \N$, there exists a sequence $\{u_{k,\ell}\}_{\ell}$ of nonradial solutions of (\ref{P-const}) for a corresponding sequence of exponents $\{p_{k,\ell}\}_\ell$ with $\lim \limits_{\ell \to \infty} p_{k,\ell} = p_k$ and with the property that $u_{k,\ell}$ converges to the unique radial positive solution of (\ref{P-const}) with $p=p_k$ as $\ell \to \infty$.
Theorem~\ref{thm:main-a-const} suggests that one of these bifurcation points corresponds to a global branch covering the unbounded interval
\[
\left(2+ \frac{2N}{\bigl(\frac{N-2}{2}\bigr)^2+ R_0^2},\infty \right)
\]
of exponents $p$. In the same way, Theorem~\ref{thm:main-a-const} complements similar bifurcation results given in \cite[Theorems~1.3 and 1.4]{gladiali-et-al2010} on the bifurcation of nonradial solutions with respect to the bifurcation parameter $R$ in the problem (\ref{P-const}) on $A=A_R:= \{x\in\mathbb{R}^N\,:\, R<|x|<R+1\}$ with $p>2$ being fixed.

Our Nehari-type variational approach allows us to characterize our solution as being a $\mathcal{K}$-ground state of problem (\ref{P}).  We are not aware of any analogous characterization for solutions of this problem in the literature.  This allows to estimate the energy of the solution and may be useful in some frameworks (for example, to detect the asymptotic behaviour of solutions as $p$ goes to infinity).

We also remark that the solutions obtained in Theorem \ref{thm:main-a-const} have Morse index greater than $N$. This follows from \cite[Theorem 1.1]{PW} and the fact that nonradial functions in $\mathcal{K}$ are axially symmetric but not foliated Schwarz symmetric.  

We briefly comment on the proof of Theorem~\ref{thm:main-a-const}. As a consequence of Theorem~\ref{thm:main_existence}, it suffices to show that the unique radial solution of (\ref{P-const}) cannot be a $\mathcal{K}$-ground state solution under the given assumptions. 
For this we show that, under assumption \eqref{eq:suff_p},  a specific instability property of the unique radial solution of (\ref{P-const}) with respect to the cone  $\mathcal{K}$ holds, see Proposition~\ref{radial-weakly-stable} below. 

\medskip

The paper is organized as follows. In Section \ref{sec2}, we collect some preliminary results and a priori estimates in the cone $\mathcal K$, see in particular Lemma \ref{le:est2}: it is interesting to observe that its proof uses trace inequalities and embeddings theorems for fractional Sobolev spaces. In Section \ref{sec3}, we prove the main result of the paper, Theorem \ref{thm:main_existence}. 
Finally, in Section \ref{sec:case-constant-a}, we deal with the case $a\equiv \mathrm{const.}$ and prove Theorem \ref{thm:main-a-const}.

\section{Preliminary results}\label{sec2}

\subsection{The linear problem in the cone} 

Let us consider the set $\mathcal{K}$ defined in \eqref{eq:K_tilde_def}.
It is easy to verify that it is a closed convex cone in $C^1_0(A)$, that is:
\begin{itemize}
\item[(i)] if $u\in \mathcal{K}$ and $\lambda>0$ then $\lambda u\in \mathcal{K}$;
\item[(ii)] if $u,v\in \mathcal{K}$ then $u+v\in \mathcal{K}$;
\item[(iii)] if $u,-u\in \mathcal{K}$ then $u\equiv0$;
\item[(iv)] $\mathcal{K}$ is closed in the $C^1_0$-topology.
\end{itemize}

We begin with the following auxiliary result; it deals with the linear problem with right-hand side belonging to $\mathcal{K}$.

\begin{lemma}\label{le:invarianceK}
For any $h\in \mathcal{K}$, the unique solution $u$ to the linear problem 
\begin{equation}\label{Ph}
\begin{cases}
-\Delta u + u =h \quad&\mbox{in } A,\\
u=0&\mbox{on }\partial A
\end{cases}
\end{equation}
belongs to $\mathcal{K}$.
\end{lemma}
\begin{proof}
Let us first notice that, since $h \in C^1_0(A) \subset C^{0,\alpha}(\overline A)$, by elliptic regularity we have $u \in C^{2,\alpha}(\overline A)$ and thus $u \in C^1_0(A)$, as well.
By uniqueness and thanks to the fact that the problem (i.e., the operator, the right-hand side, and the domain) is invariant under the action of the group $O(N-1)\times O(1)$, the solution $u$ is such that $u=\mathfrak{u}(r,\theta)$ and $\mathfrak{u}(r,\theta)=\mathfrak{u}(r,-\theta)$ for $\theta\in(-\pi/2,\pi/2)$. Furthermore, by the maximum principle, since $h\ge0$, also $u\ge 0$ in $A$. 

In order to prove the monotonicity with respect to $\theta$, we recall the expression of the Laplacian of an axially symmetric function:
\begin{equation}\label{eq:Laplace_radial}
\Delta u=\mathfrak{u}_{rr}+\frac{N-1}{r}\mathfrak{u}_r+\frac{1}{r^2}\Delta_{\mathbb S^{N-1}}\mathfrak{u},
\end{equation}
where $\Delta_{\mathbb S^{N-1}}$ is the Laplace-Beltrami operator on the $(N-1)$-sphere with its canonical metric, which for axially symmetric functions has the form 
$$
\Delta_{\mathbb S^{N-1}}\mathfrak{u}=\frac{1}{\cos^{N-2}\theta}\partial_\theta\left(\cos^{N-2}\theta \, \mathfrak{u}_\theta \right).
$$
Therefore, if we perform a partial derivative in $\theta$ for the equation in \eqref{Ph}, we get the pointwise equation
$$
-\mathfrak{u}_{\theta rr}-\frac{N-1}{r} \mathfrak{u}_{\theta r}+\frac{N-2}{r^2}\frac{1}{\cos^2\theta}\mathfrak{u}_\theta+\frac{N-2}{r^2}\tan\theta \, \mathfrak{u}_{\theta\theta}-\frac{1}{r^2}\mathfrak{u}_{\theta\theta\theta}+\mathfrak{u}_\theta=\mathfrak h_\theta,
$$
for $(r,\theta)\in (R_0,R_1)\times\left(-\pi/2,\pi/2\right)$.
Defining $u_\theta(x), h_\theta(x)$ by the relations
\[
u_\theta(x)=\mathfrak{u}_\theta\left(|x|,\arcsin\left(\frac{x_N}{r}\right)\right),
\qquad
h_\theta(x)=\mathfrak{h}_\theta\left(|x|,\arcsin\left(\frac{x_N}{r}\right)\right),
\]
and noticing that $r^2\cos^2\theta=x_1^2+\ldots+x_{N-1}^2$,
we can rewrite the previous expression as
$$
-\Delta u_\theta+\Bigl(\frac{N-2}{x_1^2+\ldots+x_{N-1}^2}+1\Bigr)u_\theta = h_\theta 
\quad\mbox{in } \tilde A:= \{ x \in A\::\: x_1^2+\ldots+x_{N-1}^2 \not = 0\}.
$$
We now wish to show that
\begin{equation}
  \label{eq:claim-lemma-2-1}
u_\theta \le 0 \qquad \text{in }\tilde A_+:= \tilde A\cap \{x_N>0\}.
\end{equation}
For this we first note that $u_\theta= 0$ on $\partial \tilde A_+$. Indeed,   
being $u$ regular and  axially symmetric with respect to the $x_N$-axis, we deduce that $\mathfrak{u}_\theta(r,\pi/2)=0$ for every $r\in [R_0,R_1]$, so that $u_\theta$ vanishes along the $x_N$-axis. Moreover, since $u=0$ on $\partial A$, and $\mathfrak{u}_\theta(r,\theta)$ is the derivative of $u$ in the tangential direction to $\partial A$, we have $\mathfrak{u}_\theta(r,\theta)=0$ on $\{R_0,R_1\}\times (0,\pi/2)$, so that $u_\theta=0$ on $\partial A$. Finally, since  $u$ is axially symmetric and even with respect to $\theta$, again by regularity we have $\mathfrak{u}_\theta(r,0^+)=-\mathfrak{u}_\theta(r,0^-)=\mathfrak{u}_\theta(r,0)$ for every $r\in [R_0,R_1]$, and so $\mathfrak{u}_\theta(r,0)=0$ for every $r\in [R_0,R_1]$. Hence $u_\theta= 0$ on $\partial (A\cap \{x_N>0\})$.

In sum, $u_\theta$ satisfies the problem
\begin{equation}\label{pb:for-u_theta}
\left\{  \begin{aligned}
-\Delta u_\theta+\Bigl(\frac{N-2}{x_1^2+\ldots+x_{N-1}^2}+1\Bigr)u_\theta &= h_\theta\le 0 &&\quad \mbox{pointwisely in } \tilde A_+\\
 u_\theta &= 0 &&\quad   \text{on $\partial \tilde A_+$.}   
  \end{aligned}
\right.
\end{equation}
Due to the singularity of the equation on the $x_N$-axis, we cannot apply the weak maximum principle directly to deduce that $u_\theta \le 0$. Instead, we let $\varepsilon>0$ and consider the function $v:= (u_\theta-\varepsilon)^+$ on $\tilde A_+$. By  (\ref{pb:for-u_theta}) and since $u_\theta \in C^1(\tilde A_+)$, the function $v \in  H^1_0(\tilde A_+)$ has compact support in $\tilde A_+$. Hence we may multiply (\ref{pb:for-u_theta}) with $v$ and integrate by parts, obtaining the inequality
\begin{align}
  &\int_{\tilde A_+}\Bigl(|\nabla v|^2 +\Bigl(\frac{N-2}{x_1^2+\ldots+x_{N-1}^2}+1\Bigr)v^2\Bigr)\,dx \nonumber\\
  &\le \int_{\tilde A_+}\Bigl(\nabla u_\theta \cdot \nabla v   +\Bigl(\frac{N-2}{x_1^2+\ldots+x_{N-1}^2}+1\Bigr)u_\theta v\Bigr)\,dx \nonumber \\
  &= \int_{\tilde A_+}\Bigl( -\Delta u_\theta+\Bigl(\frac{N-2}{x_1^2+\ldots+x_{N-1}^2}+1\Bigr)u_\theta\Bigr) v\,dx = \int_{\tilde A_+} h_\theta v \,dx \le 0. \label{integral-estimate}
\end{align}
Here, the integration by parts in the second step is justified by approximating $v$ in the $H^1$-norm by a sequence of functions $(v_n)_n \subset C^\infty_c(\Omega)$, where $\Omega \subset \tilde A_+$ is a compactly contained subdomain containing the support of $v$. Now (\ref{integral-estimate}) implies that $v = (u_\theta-\varepsilon)^+ \equiv 0$ in $\tilde A_+$. Since $\varepsilon>0$ was chosen arbitrarily, we deduce (\ref{eq:claim-lemma-2-1}), as claimed. Since $\mathfrak{u}$ is even with respect to $\theta$, then $\mathfrak{u}_\theta\ge 0$ for $\theta\in(-\pi/2,0)$ and the proof is concluded.
\end{proof}

\subsection{A priori estimates in the cone}

In this section, we deal with the larger cone
$$
\widetilde{\mathcal K}:= \left\{u\in H^1_0(A)\,:\, \begin{aligned} & u=\mathfrak{u}(r,\theta), \,u\ge 0 \mbox{ a.e. in } A,\\
&\mathfrak{u}(r,\theta)=\mathfrak{u}(r,-\theta) \mbox{ a.e. in }(R_0,R_1)\times(0,\pi/2) ,\\ 
& \mathfrak{u}_\theta(r,\theta)\le 0 \mbox{ a.e. in }(R_0,R_1)\times(0,\pi/2). 
\end{aligned}
\right\}.
$$
where $\mathfrak{u}_\theta$ denotes the weak derivative of $\mathfrak{u}$ with respect to $\theta$. Notice that, being $u\in H^1(A)$, we have that $\mathfrak{u}\in H^1_{\mathrm{loc}}((R_0,R_1)\times(-\pi/2,\pi/2))$ (see for example \cite[Proposition 9.6]{brezis2010}), so that $\mathfrak{u}_\theta \in L^2_{\mathrm{loc}}((R_0,R_1)\times(-\pi/2,\pi/2))$ and the almost everywhere sign condition on $\mathfrak{u}_\theta$ appearing in \eqref{cone} makes sense.

Of course, $\mathcal{K} = \widetilde{\mathcal K} \cap C^1_0(A)$. The fact that $\widetilde{\mathcal K}$ is a cone is easily verified (cf (i)-(iii) at the beginning of the previous section). Below, we explicitly prove that $\widetilde{\mathcal K}$ is closed with respect to the $H^1$-topology.

\begin{lemma}\label{le:Kcone}
$\widetilde{\mathcal K}$ is closed with respect to the $H^1(A)$-norm; as a consequence, it is weakly closed.
\end{lemma}
\begin{proof}
Let $\{u_n\}_{n}\subset\widetilde{\mathcal K}$ and $u\in H^1_0(A)$ be such that $u_n\to u$ in $H^1(A)$ as $n\to\infty$. Clearly, $u$ is axially symmetric, non-negative and even with respect to $\theta$ by pointwise almost everywhere convergence up to a subsequence. Let us check that $\mathfrak{u}_\theta\le 0$. Again by \cite[Proposition 9.6]{brezis2010}, we can write
\[
0\geq \frac{\partial\mathfrak{u}_n}{\partial\theta} = \nabla u_n \cdot \frac{\partial x}{\partial \theta} \to \nabla u \cdot \frac{\partial x}{\partial \theta} =u_\theta
\]
almost everywhere, as $n\to\infty$. 
Then, $u \in \widetilde{\mathcal K}$, proving that $\widetilde{\mathcal K}$ is closed in the strong $H^1$-topology. Since a cone is a convex set, we conclude that $\widetilde{\mathcal K}$ is weakly closed, as well.
\end{proof}

In the following we denote $D:=A\cap\{x_N=0\}$ and we use the notation $x=(x',x_N)$ for every $x\in \mathbb R^N$ so that $x=(x',0)$ for every $x\in D$. The main result of this section is the following a priori bound for functions in the cone $\widetilde{\mathcal K}$.

\begin{lemma}\label{le:est2}
$\widetilde{\mathcal K} \subset L^q(A)$ for every $q\geq1$. Moreover, for every $q\ge 1$ there exists a positive constant $C(q)$ such that 
\begin{equation}\label{eq:est3}
\|u\|_{L^q(A)}\leq C(q) \|u\|_{H^1(A)} \quad\mbox{for every }u\in \widetilde{\mathcal K}.
\end{equation}
\end{lemma}
\begin{proof} Let $\tilde u$ be the trivial extension to zero of $u$ outside $A$. 
On the one hand, by the trace inequality, we have
\begin{equation}\label{RtoA}
\left[\tilde u(x',0) \right]_{H^{1/2}(\mathbb R^{N-1})}\le C\|\tilde u\|_{H^1(\mathbb R^N_+)}\le C\|\tilde u\|_{H^1(\mathbb R^N)}= C\|u\|_{H^1(A)}.
\end{equation}
On the other hand, by \cite[Lemma 5.1]{DN} with $n=N-1$, $s=1/2$, since $\tilde u(x',0)$ is a radial function in $\mathbb R^{N-1}$, we obtain for every $c>-1$
\begin{equation}\label{DeNapoli}
\left(\int_{\mathbb R^{N-1}}|x'|^c\left|\tilde u(x',0)\right|^{2^*_c}dx'\right)^{1/2^*_c}\le C [\tilde u(x',0)]_{H^{1/2}(\mathbb R^{N-1})},
\end{equation}
where $2^*_c:=\frac{2(N-1+c)}{N-2}$. Notice that \cite[Lemma 5.1]{DN} is stated for $C^\infty_c(\R^{N-1})$ radial functions, but it can be extended, by a density argument, to $H^{1/2}(\R^{N-1})$ radial functions.
Moreover, being $R_0>0$,
\begin{equation}
\label{DtoR}
\int_{\mathbb R^{N-1}}|x'|^c\left|\tilde u(x',0)\right|^{2^*_c} dx'=\int_{D}|x'|^c\left| u(x',0)\right|^{2^*_c}dx' \ge R_0^c \left\|u(\cdot,0) \right\|_{L^{2^*_c}(D)}^{2^*_c}.
\end{equation}
Since $2^*_c\to\infty$ as $c\to\infty$, combining \eqref{RtoA}, \eqref{DeNapoli} and \eqref{DtoR},  we get the existence of a constant $C_1(q)$ such that
\begin{equation}\label{eq:est1}
\left\|u(\cdot,0)\right\|_{L^q(D)}\le C_1(q)\|u\|_{H^1(A)}\quad\mbox{for every }u\in \widetilde{\mathcal K}.
\end{equation}

By the axial symmetry and the monotonicity properties of $u\in \widetilde{\mathcal K}$, we deduce
\begin{multline*}
\int_A |u|^q dx =\omega_{N-2}\int_{R_0}^{R_1}\int_{-\frac{\pi}{2}}^{\frac{\pi}{2}}|\mathfrak u(r,\theta)|^q(\cos\theta)^{N-2}r^{N-1}  dr \,d\theta\\
\le \omega_{N-2}\int_{R_0}^{R_1}\int_{-\frac{\pi}{2}}^{\frac{\pi}{2}}|\mathfrak u(r,0)|^q R_1 r^{N-2}dr \,d\theta=\omega_{N-2}R_1\pi \int_{R_0}^{R_1}|\mathfrak u(r,0)|^q r^{N-2}dr\\
=R_1\pi\int_D\left|u(\cdot,0)\right|^q dx'. 
\end{multline*}
Combining the last inequality with \eqref{eq:est1} we have the desired estimate.
\end{proof}

\subsection{The fixed point operator $T$} 
Hereafter, let $p>2$ be fixed. 
We define the following operator:
\[
T:\widetilde{\mathcal K}\cup C^1_0(A)\to H^1_0(A) \qquad T(u):=(-\Delta+\mathrm{Id})^{-1}(a(x) |u|^{p-2}u),
\]
namely $T(u)=v$ is the unique $H^1_0(A)$ function satisfying
\begin{equation}\label{eq:T_def}
\int_A (\nabla v\cdot \nabla\varphi+v\varphi)\,dx=\int_A a(x)|u|^{p-2} u \varphi \,dx
\quad\text{for every } \varphi\in C^\infty_c(A).
\end{equation}
This definition is clearly well posed when $u \in C^1_0(A)$ since $a(x) |u|^{p-2}u \in C^1_0(A)$. 
On the other hand, when $u \in \widetilde{\mathcal K}$, by Lemma \ref{le:est2} we have
$$
u^{p-1} \in L^q(A) \quad\mbox{for every }q\geq 1.
$$
This, in particular, implies that
\begin{equation}\label{eq:est_Lq}
a(x)|u|^{p-2}u = a(x)u^{p-1}  \in L^2(A).
\end{equation}
so that $T(u)$ is well-defined, again, with $T(u) \in H^1_0(A)$.

We also observe that 
\begin{equation}\label{eq:T(K)}
T(\mathcal{K})\subset\mathcal{K}.
\end{equation}
Indeed, thanks to \eqref{a_assumptions}, we have that $a(x)u^{p-1}\in\mathcal{K}$ for every $u\in\mathcal{K}$ and so, by Lemma \ref{le:invarianceK}, $T(u)\in \mathcal{K}$.

We now prove that $T$, when restricted to $\mathcal{K}$, has suitable continuity and compactness properties.

\begin{proposition}\label{prop:T_compact}
Let $\{u_n\}_n\subset \mathcal{K}$ be such that $u_n\rightharpoonup u$ weakly in $H^1(A)$ for some $u\in\widetilde{\mathcal K}$. Then $T(u) \in \mathcal{K}$ and $T(u_n) \to T(u)$ in $C^1_0(A)$. 
\end{proposition}
\begin{proof}
Let us first prove that $T(u_n)\to T(u)$ in $H^1(A)$. By the definition \eqref{eq:T_def} of $T$, we have
\begin{multline*}
\|T(u_{n})-T(u)\|_{H^1(A)}^2=\int_A  a(x)(u_{n}^{p-1}-u^{p-1})(T(u_{n})-T(u))\,dx \\
\leq (p-1)\|a\|_{L^\infty(A)} \int_A (u_{n}+u)^{p-2}|u_{n}-u||T(u_{n})-T(u)|\,dx,
\end{multline*}
where we used the inequality (see \cite{Damascelli1998})
\[
\left| \xi^{p-1} - \eta^{p-1} \right| \leq (p-1)\left(\xi+\eta\right)^{p-2}|\xi-\eta|\quad\text{for every }\xi,\eta\in \R^+, \ p\geq2.
\]
Let $\alpha>1$ and $\beta>\max \left\{N,\frac{1}{p-2}\right\}$ be such that
\[
\frac{1}{\alpha}+\frac{1}{\beta}+\frac{1}{2}=1.
\]
By the choice of $\beta$, it results $\alpha<2^*$ and $\beta(p-2)>1$, and so, Lemma \ref{le:est2} and the H\"older inequality applied to the previous expression provide
\begin{multline*}
\|T(u_{n})-T(u)\|_{H^1(A)}^2 \leq \\
C \left(\|u_{n}\|_{L^{\beta(p-2)}(A)}+\|u\|_{L^{\beta(p-2)}(A)}\right)^{p-2} \|u_{n}-u\|_{L^\alpha(A)}\|T(u_{n})-T(u)\|_{H^1(A)},
\end{multline*}
with $C=(p-1)\|a\|_{L^\infty(A)}$.
Since $\{u_n\}_n$ is weakly convergent, it is bounded in the $H^1(A)$-norm. Hence, by combining relation \eqref{eq:est3} with the previous estimate we deduce
\[
\|T(u_{n})-T(u)\|_{H^1(A)} \leq C' \|u_{n}-u\|_{L^\alpha(A)},
\]
for a constant $C'$ not depending on $n$.
Hence, the weak convergence $u_n\rightharpoonup u$ in $H^1(A)$ and $\alpha<2^*$ imply 
\[
\lim_{n\to+\infty} \|T(u_{n})-T(u)\|_{H^1(A)}=0,
\]
as desired. Since $\widetilde{\mathcal K}$ is closed, $T(u)\in \widetilde{\mathcal K}$ and, by elliptic regularity, $T(u)\in \mathcal{K}$.
\smallbreak
Now, let $v_n:= T(u_n)$ for $n \in \N$ and $v:= T(u)$. By the first  part of the proof, we know that $v_n \to v$ in $H^1(A)$. Hence it suffices to show that the sequence $\{v_n\}_n$ is relatively compact in $C^1_0(A)$. Let $q \in (1,\infty)$. By (\ref{a_assumptions}), Lemma~\ref{le:est2}, and the boundedness of $\{u_n\}_n$ in $H^1(A)$, we see that
\begin{equation}
 \label{eq:RHS-bound}
\|f_n\|_{L^q(A)} \le C_q \qquad \text{for }n \in \N \text{ with }f_n:= a(x) u_n^{p-1} 
\end{equation}
and some constant $C_q>0$. Since the functions $v_n \in H^1_0(A)$ solve $-\Delta v_n + v_n = f_n$ in $A$, elliptic regularity estimates show that 
$$
\|v_n\|_{W^{2,q}(A)} \le C_q' \qquad \text{for } n \in \N
$$
with some constant $C'_q>0$. Since, as $\partial A$ is smooth, we have a compact embedding $W^{2,q}(A) \cap H^1_0(A) \hookrightarrow C^1_0(A)$ for $q>N$, we conclude that the sequence $\{v_n\}_n$ is relatively compact in $C^1_0(A)$.
\end{proof}


\subsection{An equivalent minimax characterization}
We define the minimax value
\begin{equation}\label{eq:def-c-I}
d_I:= \inf_{\stackrel{u \in \mathcal{K}}{u \not\equiv 0}}\sup_{t>0}I(tu),
\end{equation}
where the functional $I$ is defined by \eqref{eq:I_def}. 
By elementary properties of $I$, it is easy to see that for every function $u \in C^1_0(A) \setminus \{0\}$ there exists precisely one critical point $t_u>0$ of the function $t \mapsto I(tu)$ which is the global maximum of this function on $(0,\infty)$ (see for example \cite[Chapter 4]{Willem}). More precisely,
\begin{equation}\label{eq:t_u_def}
t_u=\left(\frac{\|u\|_{H^1(A)}^2}{\int_A a(x)|u|^p\,dx}\right)^{1/(p-2)}
\end{equation}
and it is straightforward to verify that $t_u u \in \mathcal{N}_{\mathcal{K}}$.

\begin{lemma}\label{le:cI=dI}
The value $c_I$ introduced in \eqref{eq:Ic-equivalent} coincides with $d_I$.
\end{lemma}
\begin{proof}
For every $u\in \mathcal{K} \setminus\{0\}$,  $t_u u \in \mathcal{N}_{\mathcal{K}}$, with $t_u$ as in \eqref{eq:t_u_def}, thus implying that 
\[
d_I=\inf_{u\in \mathcal{K}\setminus\{0\}} I(t_uu) \geq c_I.
\]
In order to prove the opposite inequality, notice that the map $H:u\in \mathcal{K}\cap \mathcal{S}^1 \mapsto t_u u \in \mathcal{N}_{\mathcal{K}}$, with $\mathcal{S}^1=\{u\in H^1_0(A): \, \|u\|_{H^1(A)}=1\}$, is bijective. Indeed, the map $v\in \mathcal{N}_{\mathcal{K}} \mapsto v/\|v\|_{H^1(A)} \in \mathcal{K}\cap \mathcal{S}^1$ is the inverse of $H$ by the uniqueness of $t_u$ and by the fact that $t_u=1$ if and only if $u\in \mathcal{N}_{\mathcal{K}}$. Therefore
\[
d_I\leq \inf_{u\in \mathcal{K}\cap \mathcal{S}^1} \sup_{t>0} I(tu)
= \inf_{u\in \mathcal{K}\cap \mathcal{S}^1} I(H(u))= c_I. \qedhere
\]
\end{proof}

\section{Proof of the main result}\label{sec3}

In this section, we give the proof of Theorem \ref{thm:main_existence} via a critical point theory approach in the space $C^1_0(A)$ and more precisely in the cone $\mathcal{K}$ introduced in \eqref{eq:K_tilde_def}. 
Although $\mathcal{K}$ is a subset of $C^1_0(A)$, we emphasize that our argument requires the use of both  the $C^1_0$ and the $H^1$ topology.
As already mentioned in the introduction, the solution is found as a fixed point of the operator $T$ given in \eqref{eq:T_def}, by means of a dynamical system point of view applied to a suitable descent flow.
For the reader's convenience, we divide the section in several subsections.

\subsection{Compactness and geometry of the functional $I$}
Let us consider the functional $I$ defined in \eqref{eq:I_def}. Recalling the definition of the operator $T$ given in \eqref{eq:T_def}, we observe that  
\[
I'(u)v= \int_A(\nabla u \cdot \nabla v + uv - a(x)|u|^{p-2}uv)\,dx=\langle u-T(u), v \rangle_{H^1(A)},
\]
for every $u,\,v\in C^1_0(A)$. 
We first show that $I$ satisfies a Palais-Smale type condition in $\mathcal{K}$, with respect to the $H^1$-norm.

\begin{lemma}\label{lemma:PS}
Let $\{u_n\}_{n}\subset\mathcal{K}$ be such that
\begin{itemize}
\item[(i)] $\{I(u_n)\}_{n}$ is bounded;
\item[(ii)] $\lim_{n\to+\infty} \|u_n-T(u_n)\|_{H^1(A)}=0$.
\end{itemize}
Then there exist a subsequence $\{u_{n_k}\}_{k}$ and $u\in \mathcal{K}$ such that
\[
\lim_{k\to+\infty} \|u_{n_k}-u\|_{H^1(A)}=0 \qquad\text{and}\qquad
u=T(u).
\]
\end{lemma}
\begin{proof}
By assumption (i) there exists a constant $C>0$ such that
\begin{multline}
C\geq I(u_n)=\left(\frac{1}{2}-\frac{1}{p}\right) \|u_n\|_{H^1(A)}^2 + 
\frac{1}{p}\left[\|u_n\|_{H^1(A)}^2-\int_A a(x) u_n^p \,dx \right] \\
=\left(\frac{1}{2}-\frac{1}{p}\right) \|u_n\|_{H^1(A)}^2 -\frac{1}{p} \langle T(u_n)-u_n,u_n\rangle_{H^1(A)} \\
\geq \left(\frac{1}{2}-\frac{1}{p}\right) \|u_n\|_{H^1(A)}^2 -\frac{1}{p} \| T(u_n)-u_n\|_{H^1(A)} \|u_n\|_{H^1(A)},
\end{multline}
for every $n\geq1$, where we also used the definition of $T$ (see \eqref{eq:T_def}) and the Cauchy-Schwarz inequality.
Now, the last inequality combined with assumption (ii) implies that the sequence $\{u_n\}$ is bounded in the $H^1(A)$-norm. We deduce the existence of a subsequence $\{u_{n_k}\}_{k}$ and $u\in H^1_0(A)$ such that $u_{n_k} \rightharpoonup u$ weakly in $H^1(A)$ as $k\to+\infty$. By Lemma \ref{le:Kcone}, $u \in \widetilde{\mathcal K}$. Then Proposition \ref{prop:T_compact} provides $T(u)\in\mathcal{K}$ and
\[
\lim_{k\to+\infty} \|T(u_{n_k})-T(u)\|_{H^1(A)}=0.
\]
In turn, using again assumption (ii), we obtain
\[
o(1)=\|T(u_{n_k})-u_{n_k}\|_{H^1(A)} =\|T(u)-u_{n_k}\|_{H^1(A)} +o(1)
\]
as $k\to+\infty$, from which we deduce both that $u_{n_k}$ converges to $u$ strongly in $H^1(A)$ and that $u=T(u)$. In particular, $u\in\mathcal{K}$.
\end{proof}

In the next lemma we prove that $I$ has a mountain pass type geometry.

\begin{lemma}
\label{isolated-minimum}
There exists $\alpha>0$ with the property that for 
\[
B_\alpha(\mathcal{K}):= \{u \in \mathcal{K} \::\: \|u\|_{H^1(A)} < \alpha\}, \qquad S_\alpha(\mathcal{K}) :=  \{u \in \mathcal{K} \::\: \|u\|_{H^1(A)} = \alpha\}
\] 
we have: 
\begin{enumerate}
\item[(i)] $I$ is nonnegative on $B_\alpha(\mathcal{K})$.
\item[(ii)] $\rho_\alpha:= \inf \limits_{u\in S_\alpha(\mathcal{K})} I(u)>0.$
\end{enumerate}
\end{lemma}

\begin{proof}
Let $u\in \mathcal{K}$. By Lemma~\ref{le:est2} with $q=p$, we get
\[
I(u)\ge \frac{1}{2}\|u\|^2_{H^1(A)}-\frac{1}{p}\|a\|_{L^\infty(A)}\|u\|^p_{L^p(A)}\ge \frac{1}{2}\|u\|^2_{H^1(A)}-\frac{C(p)^p}{p}\|a\|_{L^\infty(A)}\|u\|^p_{H^1(A)},
\]   
whence (i) and (ii) follow immediately, being $p>2$.
\end{proof}

\subsection{A descent flow in the cone}
In the following, we develop a descent flow argument inside the cone $\mathcal{K}$.
For every $v\in C^1_0(A)$, let 
\[
\Phi(v):=v-T(v),
\]
then $\Phi: C^1_0(A) \to C^1_0(A)$ is locally Lipschitz.
For every $u\in C^1_0(A)$, let $\eta(t,u)$ be the unique solution of the following initial value problem 
\begin{equation}\label{eq:eta_def}
\begin{cases}
\frac{d}{dt}\eta(t,u)=-\Phi(\eta(t,u))\\
\eta(0,u)=u,
\end{cases}
\end{equation}
defined on its maximal interval $[0,T_\mathrm{max}(u))$.

We observe that $T_\mathrm{max}(u)$ may be finite for some $u$, due to the fact that the right hand side of \eqref{eq:eta_def} is not normalised. We made this choice because a $C^1$ normalisation, that would have ensured existence of $\eta(t,u)$ for all times $t$ for every $u\in C^1_0(A)$, would have invalidated estimate \eqref{eq:decreasing-energy} below.

\begin{remark}\label{rem:continuous_dependence}
Being $\Phi$ locally Lipschitz, the solution of \eqref{eq:eta_def} depends continuously on the initial data (see for example \cite{D}).
That is, for every $u \in C^1_0(A)$, for every $\bar t < T_\mathrm{max}(u)$ and for every $\{v_n\} \subset C^1_0(A)$ such that $\| v_n - u \|_{C^1(A)} \to 0$, there exists $\bar{n} \geq 1$ such that, for every $n \geq \bar{n}$, the solution $\eta(t,v_n)$ is defined for every $t \in [0,\bar t]$ and 
\[
\sup_{t\in[0,\bar t]} \|\eta(t,v_n)-\eta(t,u)\|_{C^1(A)} \to 0, \quad \mbox{ as } n \to +\infty.
\]
\end{remark}

Let us show that the cone $\mathcal{K}$ is invariant under the action of the flow $\eta$.

\begin{lemma}
\label{positive-invariant-tilde-K}
For every $u\in \mathcal{K}$ and for every $t<T_\mathrm{max}(u)$, $\eta(t,u)\in\mathcal{K}$.
\end{lemma}
\begin{proof}
The proof is analogous to the one in \cite[Lemma 4.5]{BNW} (see also \cite{CN,CC}). We briefly sketch it below for the sake of completeness. For every $n\in\mathbb N$, we consider the approximation of the flow line $t\in [0,T_{\mathrm{max}}(u))\mapsto \eta(t,u)$ given by the Euler polygonal  $t\in [0,T_{\mathrm{max}}(u))\mapsto \eta_n(t,u)$. The vertices of such polygonal $\eta_n$ are defined by the following recurrence formula:
\[
\begin{cases}
&\eta_n(0,u)=0\\
&\eta_n(t_{i+1},u):= \eta_n(t_{i},u)-\frac{T_{\mathrm{max}}(u)}{n}\Phi(\eta_n(t_i,u))\quad\mbox{for all }i=0,\dots, n-1,
\end{cases}
\]
where $t_i:= \frac{i}{n}T_{\mathrm{max}}(u)$ for every $i=0,\dots, k$. Recalling the definition of $\Phi$, since $T$ preserves the cone $\mathcal{K}$, it is easy to prove that the vertices of the polygonal $\eta_n$ belong to $\mathcal{K}$ by convexity. Hence, again by convexity, $\eta_n([0,T_{\mathrm{max}}(u)),u)\subset\mathcal{K}$ for every $n$. Finally, being $\Phi$ locally Lipschitz, the following convergence holds for every $t\in [0,T_{\mathrm{max}}(u))$
\[
\lim_{n\to+\infty}\|\eta_n(t,u)-\eta(t,u)\|_{C^1(A)}=0.
\]
The statement then follows immediately, being $\mathcal{K}$ closed in the $C^1$-topology.
\end{proof}

In the next lemma we prove that the energy functional $I$ decreases along the trajectories $\eta(\cdot,u)$. Moreover, we give a condition on $u$ sufficient to guarantee the global existence of $\eta(\cdot,u)$ and to construct a related Palais-Smale sequence.

\begin{lemma}
\label{energy-bound-compactness}
Let $u \in C^1_0(A)$. Then we have 
\begin{equation}
  \label{eq:decreasing-energy}
\frac{d}{dt} I(\eta(t,u))= -\|\Phi(\eta(t,u))\|_{H^1(A)}^2\quad \text{for every } 
t \in (0,T_\mathrm{max}(u)).
\end{equation}
Consequently, the functional $I$ is nonincreasing along the trajectories of $\eta$. Moreover, if 
\begin{equation}\label{eq:limit-condition}
u \in \mathcal{K}\qquad \text{and}\qquad c_u:= \lim_{t \to T_\mathrm{max}(u)} I(\eta(t,u)) > -\infty,
\end{equation}
then $T_\mathrm{max}(u)= \infty$, and there exists a sequence
$\{s_n\}_{n}\subset (0,+\infty)$ such that $\lim \limits_{n\to+\infty} s_n=+\infty$ and 
\begin{equation}\label{eq:u*-T(u*)}
\lim_{n\to+\infty}\|\Phi(w_n)\|_{H^1(A)} =0
\end{equation} 
with
\begin{equation}\label{eq:u_n_def}
w_n:=\eta(s_n,u) \qquad \text{for } n\geq1.
\end{equation}
\end{lemma}  
\begin{proof}
Let $T_*:= T_\mathrm{max}(u)$.
For $t \in (0,T_*)$ we have
\begin{align*}
\frac{d}{dt} I(\eta(t,u))&=\int_A \Bigl[\bigl[\nabla \eta(t,u) \cdot \nabla \bigl(T(\eta(t,u))-\eta(t,u)\bigr)+ \eta(t,u) \bigl(T(\eta(t,u))-\eta(t,u)\bigr)\bigr]\\
&- \int_A a(x) |\eta(t,u)|^{p-2}\eta(t,u)\bigl(T(\eta(t,u))-\eta(t,u)\bigr)\Bigr]  dx\\
&= \langle \eta(t,u), T(\eta(t,u))-\eta(t,u) \rangle_{H^1(A)} 
- \langle T(\eta(t,u)), T(\eta(t,u))-\eta(t,u) \rangle_{H^1(A)} \\
&= -\|T(\eta(t,u))-\eta(t,u)\|_{H^1(A)}^2,
\end{align*}
as claimed in (\ref{eq:decreasing-energy}). 

Next we assume that \eqref{eq:limit-condition} holds. In order to prove that $T_*=\infty$ we proceed by contradiction, thus assuming $T_*<\infty$, and consequently
\begin{equation}\label{eq:T*contradiction}
\lim_{t\to T_*^-} \|\eta(t,u)\|_{C^1(A)}=+\infty.
\end{equation}
For $0 \le s < t < T_*$, we then have, using \eqref{eq:decreasing-energy},
\begin{align*}
\|\eta(t,u)&-\eta(s,u)\|_{H^1(A)} \le \int_s^t \left\|\frac{d}{d\tau}\eta(\tau,u)\right\|_{H^1(A)}\,d \tau = \int_s^t \sqrt{-\frac{d}{d\tau} I(\eta(\tau,u))}\,d\tau\\
&\le \sqrt{t-s} \, \sqrt{-\int_s^t \frac{d}{d\tau} I(\eta(\tau,u))\,d\tau}
= \sqrt{t-s}\, [I(\eta(s,u))-I(\eta(t,u))]^{\frac{1}{2}}\\
&\le  \sqrt{t-s} \, [I(u)-c_u]^{\frac{1}{2}}.
\end{align*}
Since, by our contradiction assumption, $T_*<\infty$, we deduce that for every sequence $\{t_n\}_{n} \subset (0,T_*)$ such that $t_n\to T_*^-$ as $n\to\infty$, $\{\eta(t_n,u)\}_{n} $ is a Cauchy sequence. This implies that there exists $w \in H^1_0(A)$ such that
\[
\lim_{t \to T_*} \|\eta(t,u)-w\|_{H^1(A)}=0.
\]
Consequently, by Proposition~\ref{prop:T_compact},
\[
T(w) \in C^1_0(A) \qquad \text{and} \qquad 
\lim_{t \to T_*} \|T(\eta(t,u)) - T(w)\|_{C^1(A)}=0.
\]
From this, by differentiating $e^t \eta(t,u)$, we deduce that
\begin{align*}
\eta(t,u) &= e^{-t}\Bigl(u + \int_0^t e^{s}T(\eta(s,u))\,ds\Bigr)\\
&\to e^{-T_*}\Bigl(u + \int_0^{T_*} e^{s}T(\eta(s,u))\,ds\Bigr) \qquad \text{in } C^1_0(A) \text{ as } t \to T_*.
\end{align*}
By uniqueness of the limit we have that the right hand side above coincides with $w$ 
and a posteriori it follows that $\eta(t,u) \to w$ in $C^1_0(A)$ as $t \to T_*$. This contradicts \eqref{eq:T*contradiction}, hence it follows that $T_* =\infty$. 

To show the existence of a sequence $\{w_n\}_{n}$ with the asserted properties, we argue by contradiction again and assume that there exists $t_0,\delta_0>0$ with the property that 
\[
\|\Phi(\eta(t,u))\|_{H^1(A)}\ge \delta_0 \qquad \text{for $t \ge t_0$.} 
\]
By (\ref{eq:decreasing-energy}), we then deduce that 
\[
I(\eta(t_0,u))- I(\eta(t,u)) \ge (t-t_0) \delta_0^2 \to \infty \qquad \text{as }t \to T_* = \infty,
\]
which contradicts assumption \eqref{eq:limit-condition}.
Hence there exists a sequence $\{s_n\}_{n}\subset (0,+\infty)$ with the required properties.
\end{proof}

Given Lemmas \ref{lemma:PS} and \ref{energy-bound-compactness}, it only remains to exhibit $u\in \mathcal{K}$ satisfying \eqref{eq:limit-condition} and the additional condition that the related Palais-Smale sequence does not converge to zero. This is the content of the next subsection.

\subsection{A dynamical systems point of view}

Partially inspired by \cite{B2001}, we show that the mountain pass geometry of the functional $I$ allows to construct a subset of $\mathcal{K}$ that is invariant for the flow and with the property that $I$ is strictly positive over this set, see Lemma \ref{positive-invariance-z-alpha} below. This set is defined as the boundary of a certain domain of attraction for the flow $\eta$.

Let $\alpha$ be given as in Lemma~\ref{isolated-minimum}. We define
\[
L_\alpha: = \{u \in B_\alpha(\mathcal{K}) \::\ I(u) < \rho_\alpha\}.
\]
It is not difficult to check that $L_\alpha$ is relatively open in $\mathcal{K}$ with respect to the $C^1$-norm, that is to say, for every $u\in L_\alpha$ there exists $\varepsilon>0$ such that
\begin{equation}\label{eq:Lopen}
\{v\in C^1_0(A) \::\ \|v-u\|_{C^1(A)} <\varepsilon \} \cap \mathcal{K} \subset L_\alpha.
\end{equation}
Moreover, $L_\alpha$ has the following positive invariance property.

\begin{lemma}
\label{positive-invariance-l-alpha}
For $u \in L_\alpha$, we have 
$T_\mathrm{max}(u)= \infty$ and $\eta(t,u) \in L_\alpha$ for all $t \ge 0$.  
\end{lemma}

\begin{proof}
By Lemma~\ref{positive-invariant-tilde-K}, we know that $\eta(t,u) \in \mathcal{K}$ for all $t \in (0,T_\mathrm{max}(u))$. Suppose by contradiction that there exists $t_1\in(0,T_{\mathrm{max}}(u))$ such that $\eta(t_1,u)\not\in L_\alpha$. Since, by Lemma~\ref{energy-bound-compactness}, $I(\eta(t_1,u)) \le  I(u) < \rho_\alpha$, necessarily $\|\eta(t_1,u)\|_{H^1(A)}\ge\alpha$. We observe that the map $t\in[0,T_{\mathrm{max}}(u))\mapsto \|\eta(t,u)\|_{H^1(A)}$ is continuous, by virtue of the continuous embedding $C^1(A)\hookrightarrow H^1(A)$, therefore there exists $t_0\in (0,t_1]$ such that $\eta(t_0,u)\in S_\alpha(\mathcal{K})$. This contradicts Lemma~\ref{isolated-minimum}(ii), being $I(\eta(t_0,u)) < \rho_\alpha$. Consequently,
$\eta(t,u) \in L_\alpha$ for all $t \in (0,T_\mathrm{max}(u))$, and therefore 
$\lim \limits_{t \to T_\mathrm{max}(u)} I(\eta(t,u)) \ge 0$ by Lemma~\ref{isolated-minimum}(i). 
Hence $T_\mathrm{max}(u)= \infty$ by Lemma~\ref{energy-bound-compactness}. 
\end{proof}

Next we consider the domain of attraction of $L_\alpha$ in $\mathcal{K}$, more precisely
\[
D(L_\alpha):= \{u \in \mathcal{K}\::\: \eta(t,u) \in L_\alpha \text{ for some }t \in (0,T_\mathrm{max}(u)) \}.
\]
We notice that Lemma \ref{positive-invariance-l-alpha} implies that 
\begin{equation}\label{eq:DLalpha-invariance}
\text{if } u\in D(L_\alpha) \text{ then } T_\mathrm{max}(u)= \infty \text{ and }\eta(t,u) \in D(L_\alpha) \text{ for all }t \ge 0.
\end{equation}
Moreover, Lemmas \ref{isolated-minimum} (i), \ref{energy-bound-compactness} and \ref{positive-invariance-l-alpha} provide
\begin{equation}\label{eq:dl_alpha_I_positive}
\inf_{u\in D(L_\alpha)} I(u) \geq0.
\end{equation} 

\begin{lemma}\label{lem:DLopen}
$D(L_\alpha)$ is relatively open in $\mathcal{K}$ with respect to the $C^1$-norm, that is to say, for every $u\in D(L_\alpha)$ there exists $\delta>0$ such that
\[
\{v\in C^1_0(A) \::\ \|v-u\|_{C^1(A)} <\delta \} \cap \mathcal{K} \subset D(L_\alpha).
\]
\end{lemma}
\begin{proof}
Let $u\in D(L_\alpha)$. By definition there exists $t_0\in [0,T_\mathrm{max}(u))$ such that $\eta(t_0,u)\in L_\alpha$. On the one hand, being $L_\alpha$ relatively open in $\mathcal{K}$ with respect to the $C^1$-norm (see \eqref{eq:Lopen}), there exists $\varepsilon>0$ such that
\begin{equation}\label{eq:DLopen1}
\{w\in C^1_0(A) \::\ \|w-\eta(t_0,u) \|_{C^1(A)} <\varepsilon \} \cap \mathcal{K} \subset L_\alpha.
\end{equation}
On the other hand, given such $\varepsilon>0$, by Remark \ref{rem:continuous_dependence} there exists $\delta>0$ such that
\begin{equation}\label{eq:DLopen2}
v\in C^1_0(A), \ \|v-u\|_{C^1(A)}<\delta \quad\text{implies}\quad 
\|\eta(t_0,v)-\eta(t_0,u)\|_{C^1(A)}<\varepsilon.
\end{equation}
By combining \eqref{eq:DLopen1} and \eqref{eq:DLopen2}, we deduce that such $\delta>0$ satisfies the requested properties.
\end{proof}

We denote by $Z_\alpha$ the relative boundary of $D(L_\alpha)$ in $\mathcal{K}$ with respect to the $C^1$-norm.
In view of Lemma \ref{lem:DLopen} and of the fact that $\mathcal{K}$ is closed with respect to the $C^1$-topology, we have more explicitly
\begin{equation}\label{eq:Z_alpha_def}
Z_\alpha:=\overline{D(L_\alpha)}\setminus D(L_\alpha),
\end{equation}
where $\overline{D(L_\alpha)}$ denotes the standard closure of $D(L_\alpha)$ in $C^1_0(A)$ with respect to the $C^1$-norm.

\begin{lemma}\label{lem:Z_alpha_not_empty}
The set $Z_\alpha$ defined in \eqref{eq:Z_alpha_def} is not empty. More precisely, for every $\psi\in \mathcal{K} \setminus \{0\}$ there exists $t_*>0$ such that $t_*\psi\in Z_\alpha$.
\end{lemma}
\begin{proof}
For $\psi\in \mathcal{K}\setminus\{0\}$, let
\[
J_\psi:=\{ t\geq0:\, t\psi \in D(L_\alpha)\}.
\]
On the one hand, there exists $\varepsilon>0$ such that $[0,\varepsilon)\subset J_\psi$ because $0 \in L_\alpha \subset D(L_\alpha)$ and $D(L_\alpha)$ is relatively open in $\mathcal{K}$ in virtue of Lemma \ref{lem:DLopen}. On the other hand, $J_\psi$ is bounded, as there exists $\bar t>0$ such that $I(t\psi)\le -1$ for every $t\geq\bar{t}$, which implies that $t \psi \not \in D(L_\alpha)$  for every $t\geq\bar{t}$ by virtue of \eqref{eq:dl_alpha_I_positive}. As a consequence, we have that
\[
t_*:=\sup J_\psi \in (0,\infty).
\]
Then $t_*\psi\in Z_\alpha$, by definition of $Z_\alpha$.
\end{proof}

By the continuity of the flow $\eta$ with respect to the $C^1$-norm (see Remark \ref{rem:continuous_dependence}), the following property is a consequence of Lemmas~\ref{energy-bound-compactness} and \ref{positive-invariance-l-alpha}.

\begin{lemma}\label{positive-invariance-z-alpha}
For $u \in Z_\alpha$, we have 
$T_\mathrm{max}(u)= \infty$ and 
$$
\eta(t,u) \in Z_\alpha, \quad I(\eta(t,u)) \ge \rho_\alpha
\qquad \text{for all }t \ge 0.
$$
\end{lemma}
\begin{proof}
First we notice that, for $u \in Z_\alpha$, $T_\mathrm{max}(u)= \infty$ by virtue of Lemma \ref{energy-bound-compactness} (see in particular condition \eqref{eq:limit-condition}) and of property \eqref{eq:dl_alpha_I_positive}.

Next we prove that, if $u\in Z_\alpha$, $\eta(t,u)\in Z_\alpha$ for every $t>0$.
To this aim, suppose by contradiction that there exists $t_0>0$ such that $\eta(t_0,u)\in D(L_\alpha)\cup (\mathcal{K} \setminus\overline{D(L_\alpha)})$. If $\eta(t_0,u)\in D(L_\alpha)$, by definition of $D(L_\alpha)$, there exists $t_1\in(t_0,T_\mathrm{max}(u))$ such that $\eta(t_1,u)\in L_\alpha$. This means that $u\in D(L_\alpha)$, which is impossible by definition of $Z_\alpha$. 
It remains to rule out the possibility that $\eta(t_0,u)\in\mathcal{K}\setminus\overline{D(L_\alpha)}$. 
Being $\mathcal{K}\setminus\overline{D(L_\alpha)}$ relatively open in $\mathcal{K}$, there exists $\varepsilon=\varepsilon(t_0)>0$ such that 
\begin{equation}\label{eq:B_eps}
v\in \mathcal{K}, \ \|v-\eta(t_0,u)\|_{C^1(A)} <\varepsilon
\quad\text{implies}\quad
v\in  \mathcal{K}\setminus \overline{D(L_\alpha)}.
\end{equation}
Now, since $u\in Z_\alpha$, there exists a sequence $\{v_n\}_{n}$ with the property that
\begin{equation}\label{eq:v_n_def}
v_n\in D(L_\alpha) \text{ for every } n\in\N, \qquad
\lim_{n\to+\infty}\|v_n-u\|_{C^1(A)}=0.
\end{equation}
Therefore, by Remark \ref{rem:continuous_dependence}, given $\varepsilon$ as in \eqref{eq:B_eps}, there exists $n_0\in\mathbb N$ such that  
\begin{equation}\label{eq:invariance-z-alpha1}
\|\eta(t_0,v_n)-\eta(t_0,u)\|_{C^1(A)}<\varepsilon \quad\text{for every } n\ge n_0.
\end{equation}
By combining \eqref{eq:B_eps} and \eqref{eq:invariance-z-alpha1}, we infer that 
\begin{equation}\label{eq:etat0vn}
\eta(t_0,v_n)\in \mathcal{K}\setminus\overline{D(L_\alpha)}\quad\mbox{ for every } n\ge n_0.
\end{equation} 
On the other hand, since $\{v_n\}\subset D(L_\alpha)$, $\eta(t_0,v_n)\in D(L_\alpha)$ for every $n$ (see \eqref{eq:DLalpha-invariance}). This contradicts \eqref{eq:etat0vn} and concludes this part of the proof. 

Let us prove the third property, that is to say, if $u\in Z_\alpha$ then $I(\eta(t,u)) \ge \rho_\alpha$ for all $t \ge 0$. We proceed again by contradiction. Let $\bar{t}\geq0$ be such that 
\begin{equation}\label{eq:contradiction_hp}
I(\eta(\bar{t},u)) <\rho_\alpha
\end{equation}
Being $u\not\in D(L_\alpha)$, we deduce that $\|\eta(\bar{t},u)\|_{H^1(A)}\geq\alpha$. From the definition of $\rho_\alpha$ we infer that indeed
\begin{equation}\label{eq:greater_alpha}
\|\eta(\bar{t},u)\|_{H^1(A)} >\alpha.
\end{equation}
Now, let $\{v_n\}$ be as in \eqref{eq:v_n_def}. On the one hand, \eqref{eq:greater_alpha} and the continuous dependence of $\eta$ on the initial data (see Remark \ref{rem:continuous_dependence}) imply the existence of $\bar{n}\in\N$ such that
\begin{equation}\label{eq:greater_alpha2}
\|\eta(\bar{t},v_n)\|_{H^1(A)} >\alpha \qquad\text{for every } n\geq\bar{n}.
\end{equation}
On the other hand, since $\{v_n\} \subset D(L_\alpha)$ for every $n\in\N$, there exists a sequence $\{t_n\}\subset [0,+\infty)$ such that
\[
\|\eta(t_n,v_n)\|_{H^1(A)}<\alpha \quad\text{and}\quad 
I(\eta(t_n,v_n))<\rho_\alpha, \qquad\text{for every } n\in\N.
\]
Then Lemma \ref{positive-invariance-l-alpha} provides
\begin{equation}\label{eq:greater_alpha3}
\|\eta(t,v_n)\|_{H^1(A)}<\alpha \quad\text{and}\quad 
I(\eta(t,v_n))<\rho_\alpha, \qquad\text{for every } t\geq t_n, n\in\N.
\end{equation}
From \eqref{eq:greater_alpha2} and \eqref{eq:greater_alpha3} we deduce that $\bar{t}<t_n$ for every $n\in\N$, and that, for every $n\geq\bar{n}$ there exists $s_n\in (\bar{t},t_n)$ such that $\|\eta(s_n,v_n)\|_{H^1(A)}=\alpha$ for every $n\geq\bar{n}$. By definition of $\rho_\alpha$, we have $I(\eta(s_n,v_n))\geq\rho_\alpha$ for every $n\geq\bar{n}$. Being $s_n\geq \bar{t}$, Lemma \ref{energy-bound-compactness} provides $I(\eta(\bar{t},v_n))\geq\rho_\alpha$ for every $n\geq\bar{n}$. Passing to the limit (see Remark \ref{rem:continuous_dependence}) we infer that $I(\eta(\bar{t},u))\geq\rho_\alpha$, which contradicts \eqref{eq:contradiction_hp}.
\end{proof}


\subsection{Proof of Theorem \ref{thm:main_existence}}
\begin{proof}[Proof of Theorem \ref{thm:main_existence}]
Let $\psi\in \mathcal{K}\setminus\{0\}$ and let $u:= t_* \psi\in Z_\alpha$, with $t_*$ as in Lemma \ref{lem:Z_alpha_not_empty}. By Lemma~\ref{energy-bound-compactness} and Lemma~\ref{positive-invariance-z-alpha} we have that $T_\mathrm{max}(u)= \infty$ and that there exists a sequence
$\{s_n\}_{n}\subset (0,+\infty)$ such that $\lim \limits_{n\to+\infty} s_n=+\infty$ and\begin{equation}\label{eq:u*-T(u*)-1}
\lim_{n\to+\infty}\|\Phi(w_n)\|_{H^1(A)} =0,
\end{equation} 
for the sequence $\{w_n\}_n$ defined in (\ref{eq:u_n_def}).
By Lemma~\ref{lemma:PS}, we may pass to a subsequence such that 
$w_n \to w$ in $H^1(A)$ for some $w \in \mathcal{K}$ and 
$T(w) =  w$.
Lemma \ref{positive-invariance-z-alpha} provides
\begin{equation}\label{eq:w_nontrivial}
\|w\|_{H^1(A)}^2 =\lim_{n\to+\infty} \|w_n\|_{H^1(A)}^2
\geq 2 \liminf_{n\to+\infty} I(w_n) \geq 2\rho_\alpha,
\end{equation}
thus implying that $w$ is nontrivial.
Consequently, $w$ is a nontrivial solution of \eqref{P} belonging to ${\mathcal N}_{\mathcal{K}} \subset \mathcal{K}$. 

Next, we assume in addition that the function $\psi$ above satisfies
$$
\psi \in {\mathcal N}_{\mathcal{K}} \qquad \text{and}\qquad I(\psi)= c_I.
$$
Here $c_I$ is defined in~\eqref{eq:Ic-equivalent}, so $\psi$ is a minimizer of $I$ on ${\mathcal N}_{\mathcal{K}}$. In this case the function $w \in {\mathcal N}_{\mathcal{K}}$ found above satisfies
\begin{equation}\label{eq:chain-ineq}
c_I \le I(w) \le I(t_* \psi) \le I(\psi) = c_I,
\end{equation}
where in the first inequality we used that $w \in \mathcal{N}_{\mathcal{K}}$ and in the third we used that $\sup_{t>0} I(t\psi)=I(1\psi)$, being $\psi\in {\mathcal N}_{\mathcal{K}}$, cf. \eqref{eq:t_u_def}. As for the second inequality, since $w=\lim_n w_n$ in $H^1(A)$, $\|w_n\|_{L^p(A)}\to \|w\|_{L^p(A)}$ by Lemma \ref{le:est2}. By the properties of $a$, the norm $(\int_A a(x)|\cdot|^pdx)^{1/p}$ is equivalent to $\|\cdot\|_{L^p(A)}$, hence $I(w)=\lim_n I(w_n)$. Thus, the second inequality in \eqref{eq:chain-ineq} is obtained recalling that $I(w_n)=I(\eta(s_n,t_* \psi))\le I(t_* \psi)$ for every $n$ and passing to the limit in $n$. 
So, equality holds in all of the inequalities in \eqref{eq:chain-ineq}. In particular, being $I(t_*\psi)=I(\psi)$ and $\psi \in {\mathcal N}_{\mathcal{K}}$, we obtain $t_*=1$.
Hence, $I(w_n)=I(\eta(s_n,\psi))\le I(\psi)=c_I$ for every $n$. On the other hand,  by  (\ref{eq:decreasing-energy}),
$I(\eta(s_n,\psi))\searrow c_I$. Therefore, $I(\eta(s_n,\psi))=c_I$ for every $n$, and so, by the monotonicity of $I(\eta(\cdot, \psi))$ and since $\lim_n s_n=+\infty$, it follows that $I(\eta(t,\psi))= c_I$ for all $t \in (0,\infty)$.  
Therefore, by (\ref{eq:decreasing-energy}),
$$
\Phi(\eta(t,\psi))= 0 \qquad \text{for all $t \in (0,\infty)$.}
$$
Consequently, $T(\eta(t,\psi))= \eta(t,\psi)$ for all $t \in (0,\infty)$. Passing to the limit $t \mapsto 0^+$ and using the continuity of $T$, we deduce that $T(\psi)= \psi$, hence $\psi$ is a nontrivial solution of \eqref{P} belonging to $\mathcal{K}$.

To finish the proof of Theorem~\ref{thm:main_existence}, we still have to show that the minimal value $c_I$ of the functional $I$ is positive and attained on the set ${\mathcal N}_{\mathcal{K}}$. For this we let $\{\psi_\ell\}_\ell$ be a sequence in ${\mathcal N}_{\mathcal{K}}$ with the property that
$$
I(\psi_\ell) \to c_I \qquad \text{as $\ell \to \infty$.}
$$
We let $t_*^{\ell}$ be given as in Lemma \ref{lem:Z_alpha_not_empty} corresponding to $\psi_\ell$. Repeating the argument above for every $\ell$ yields corresponding nontrivial solutions $w^\ell \in {\mathcal N}_{\mathcal{K}}$ of \eqref{P} satisfying 
$$
c_I \le I(w^\ell) \le I(t_*^\ell \psi_\ell) \le I(\psi_\ell) = c_I + o(1) \qquad \text{as $\ell \to \infty$.}
$$
Notice that, by \eqref{eq:w_nontrivial} with $w=w^\ell$, we know
\begin{equation}\label{eq:cI>0}
\|w^\ell\|_{H^1(A)}^2 \geq 2 \rho_\alpha, \quad \text{for every } \ell.
\end{equation}
Since
$$
c_I + o(1) = I(w^\ell)= I(w^\ell)- \frac{1}{p}I'(w^\ell)w^\ell = \Bigl(\frac{1}{2}-\frac{1}{p}\Bigr)\|w^\ell\|^2_{H^1(A)} \qquad \text{as }\ell \to \infty,
$$
the sequence $\{w^\ell\}_\ell$ is bounded in $H^1_0(A)$. Passing to a subsequence, we may assume that $w^\ell \rightharpoonup \bar w$ in $H^1_0(A)$. Since ${\widetilde{\mathcal K}}$ is weakly closed (see Lemma \ref{le:Kcone}), we have $\bar w \in {\widetilde{\mathcal K}}$. Moreover, by Proposition~\ref{prop:T_compact}, we have
$$
\|w^\ell -T(\bar w) \|_{H^1(A)} = \|T(w^\ell) -T(\bar w) \|_{H^1(A)} \to 0 \qquad \text{as $\ell \to \infty$,}
$$
so $w^\ell \to T(\bar w)$ strongly in $H^1_0(A)$ as $\ell \to \infty$. By uniqueness of the weak limit, $\bar w = T(\bar w)$, and therefore $w^\ell \to \bar w$ strongly in $H^1_0(A)$. From this we deduce, by Proposition~\ref{prop:T_compact} that $\bar w \in \mathcal{K}$ and that
$$
w^\ell \to \bar w \qquad \text{in $C^1_0(A)$ as $\ell \to \infty$.}
$$
Consequently, $\bar w$ is a critical point of $I$ with $I(\bar w)= \lim \limits_{\ell \to \infty}I(w^\ell) = c_I>0$, the last inequality coming from \eqref{eq:cI>0}. In particular, $\bar w \not\equiv 0$, so $\bar w \in {\mathcal N}_{\mathcal{K}}$. 
Hence the minimal value $c_I$ is attained by the functional $I$ in ${\mathcal N}_{\mathcal{K}}$. 
\end{proof}

\begin{remark}
Notice that the existence of a nontrivial solution of \eqref{P} follows already from \eqref{eq:w_nontrivial}. The remaining part of the proof of Theorem \ref{thm:main_existence} gives a variational characterization that will be useful in the next section to prove the non-radiality of the solution when $a$ is constant and some additional assumptions on $p$ or $A$ hold. 
\end{remark}

\section{The case of constant $a$}
\label{sec:case-constant-a}

In this section we treat problem~\eqref{P-const} where the weight function $a$ in (\ref{P}) satisfies $a \equiv 1$. We recall that, for every fixed $p>2$, (\ref{P-const}) admits a unique positive radial solution $u_\mathrm{rad} \in C^1_0(A)$ by \cite{T}. We continue using the notation introduced in the previous sections in the special case $a \equiv 1$. In the next proposition we collect some properties satisfied by $u_\mathrm{rad}$ which will be useful in the sequel.

\begin{proposition}\label{prop:urad-prop}
Let $P:=\{u\in C^1_0(A)\,:\,u\ge 0\}$. The radial solution $u_\mathrm{rad}$ belongs to the interior of $P$ with respect to the $C^1$-norm. Moreover, the following inequalities hold
\begin{equation}\label{eq:Iturad}
I(u_\mathrm{rad})\ge I(t u_\mathrm{rad})\quad\mbox{for every } t\ge 0
\end{equation} 
and
\begin{equation}\label{eq:I'turad}
I'\bigl(tu_\mathrm{rad}\bigr)u_\mathrm{rad}>0>I'\bigl(t'u_\mathrm{rad}\bigr)u_\mathrm{rad}\quad\mbox{for every }t\in(0,1) \mbox{ and } t'\in(1,\infty).
\end{equation}
\end{proposition}
\begin{proof}
Clearly $u_\mathrm{rad}\in P$, moreover, by the Hopf Lemma, $u_\mathrm{rad}$ is contained in the interior of $P$ with respect to the $C^1$-norm.
Now, since $u_\mathrm{rad}$ is a solution of \eqref{P-const}, $u_\mathrm{rad}\in{\mathcal{N}}_{\mathcal{K}}$ and so $t_{u_\mathrm{rad}}=1$, cf. \eqref{eq:t_u_def}. Thus, the function $t\in [0,\infty)\mapsto I(tu_\mathrm{rad})$ admits a unique maximum in $t=1$, and so \eqref{eq:Iturad} follows. Moreover, the same function $t\mapsto I(tu_\mathrm{rad})$ is strictly increasing in $(0,1)$ and strictly decreasing in $(1,\infty)$, this implies \eqref{eq:I'turad} and concludes the proof.
\end{proof}

Our main tool to prove the existence of nonradial solutions of (\ref{P-const}) will be the following criterion related to instability with respect to specific directions. 

\begin{proposition}
  \label{radial-weakly-stable}
  Suppose that there exists an axially symmetric function $v \in C^1_0(A)$, written in polar coordinates as $v = \mathfrak{v}(r,\theta)$, satisfying the following properties:
  \begin{align}
    &I''(u_\mathrm{rad})(v,v)<0;  \label{radial-weakly-stable-assumption-1}
\\
    &\partial_\theta \mathfrak{v}(r,\theta) \le 0 \qquad \text{for $(r,\theta) \in (R_0,R_1) \times (0,\pi/2)$}; \label{radial-weakly-stable-assumption-2}\\
    &\mathfrak{v}(r,\theta)=\mathfrak{v}(r,-\theta) \qquad \text{for $(r,\theta) \in (R_0,R_1)\times(0,\pi/2)$;} \label{radial-weakly-stable-assumption-3}\\
    &\int_{\mathbb S^{N-1}}\mathfrak v(r,\cdot)\,d\sigma = 0 \qquad \text{for every $r \in (R_0,R_1)$,}\label{radial-weakly-stable-assumption-4}
  \end{align}
where, in the last relation, the two-variable function $\mathfrak{v}(r,\theta)$ is meant as an $N$-variable function $\mathfrak{v}(r,\theta,\phi_1,\dots,\phi_{N-2})$ which is constant with respect to $\phi_1,\dots,\phi_{N-2}$. 
Then we have
\begin{equation}
  \label{eq:c-I-smaller-U-u-rad}
  c_I < I(u_\mathrm{rad}),
\end{equation}
so every minimizer $u \in {\mathcal N}_{\mathcal{K}}$ of $I \Big|_{{\mathcal N}_{\mathcal{K}}}$ is nonradial. 
\end{proposition}

\begin{proof}
By assumption~(\ref{radial-weakly-stable-assumption-1}) and the continuity of $I''$, there exist $\delta \in (0,1)$ and $\rho>0$ with the property that 
\begin{equation}
  \label{second-der-stable-pointw-est}
\qquad  \qquad I''\bigl(t(u_\mathrm{rad}+\tau v)\bigr)(v,v)<0 \qquad \text{for $t \in
[1-\delta,1+\delta]$, $\tau \in [-\rho,\rho]$.}
\end{equation}
Since, by Proposition \ref{prop:urad-prop}, $u_\mathrm{rad}$ is contained in the interior of $P$ with respect to the $C^1$-norm, we may also assume, by adjusting $\delta$ and $\rho$ if necessary, that
$$
t(u_\mathrm{rad}+\tau v) \ge 0  \quad \text{in $A$} \qquad \text{for $t \in
[1-\delta,1+\delta]$, $\tau \in [-\rho,\rho]$.}
$$
Combining this information with assumptions~(\ref{radial-weakly-stable-assumption-2}) and (\ref{radial-weakly-stable-assumption-3}), we deduce that 
$$
t(u_\mathrm{rad}+\tau v) \in {\widetilde{\mathcal K}}\qquad \text{for $t \in
[1-\delta,1+\delta]$, $\tau \in [-\rho,\rho]$.}
$$
Moreover, since, by \eqref{eq:I'turad}, $I'\bigl((1-\delta)u_\mathrm{rad}\bigr)u_\mathrm{rad}>0>I'\bigl((1+\delta)u_\mathrm{rad}\bigr)u_\mathrm{rad}$, there exists $s \in (0,\rho)$ with  
$$
I'\bigl((1-\delta)(u_\mathrm{rad}+sv)\bigr)(u_\mathrm{rad}+sv)>0>I'\bigl((1+\delta)(u_\mathrm{rad}+sv)\bigr)(u_\mathrm{rad}+sv)
$$
By the intermediate value theorem, there exists 
$t \in [1-\delta,1+\delta]$ with 
$$
I'\bigl(t(u_\mathrm{rad}+sv)\bigr)(u_\mathrm{rad}+sv)=0 \qquad \text{and therefore} \qquad u_*:= t(u_\mathrm{rad}+sv) \in {\mathcal N}_{\mathcal{K}}.
$$
Moreover, since $u_\mathrm{rad} \in {\mathcal N}_{\mathcal{K}}$, by a Taylor expansion, \eqref{eq:Iturad} and (\ref{second-der-stable-pointw-est}) we have 
\begin{align*}
I(u_*) -I(u_\mathrm{rad})&\le I(u_*) -I(tu_\mathrm{rad})\\
&= s t I'(t\, u_\mathrm{rad})v +  t^2 \int_0^s
                                               I''(t(u_\mathrm{rad}+\tau v))(v,v)(s-\tau) \:d\tau\\
&< s t I'(t\, u_\mathrm{rad})v =st\left( t \langle u_\mathrm{rad}, v \rangle_{H^1(A)} - t^{p-1} \int_{A}u_\mathrm{rad}^{p-1} v\,dx \right)\\
  &= st^2(1-t^{p-2})\int_{R_0}^{R_1}r^{N-1}u_\mathrm{rad}^{p-1}(r) \int_{\mathbb S^{N-1}}\mathfrak v(r,\cdot)\,d\sigma dr = 0 
\end{align*}
where we used assumption~(\ref{radial-weakly-stable-assumption-4}) in the last step. Consequently, $c_I \le I(u_*)<I(u_\mathrm{rad})$, as claimed in (\ref{eq:c-I-smaller-U-u-rad}).
\end{proof}

To find a function $v \in C^1_0(A)$ satisfying the assumptions of Proposition~\ref{radial-weakly-stable}, 
we take inspiration from \cite{gladiali-et-al2010}.

\begin{lemma}\label{lem:v_def}
Let $\alpha_1$ be the first eigenvalue of the one dimensional weighted eigenvalue problem
\begin{equation}
  \label{eq:radial-eigenvalue-problem}
  \left\{
    \begin{aligned}
      &-w_{rr} -\frac{N-1}{r}w_r +\bigl(1 - (p-1)u_\mathrm{rad}^{p-2}\bigr)w = \frac{\alpha}{r^2}w &\quad \text{in $(R_0,R_1)$,}\\
      &\quad w(R_0)=w(R_1)= 0,
   \end{aligned}
 \right.
\end{equation}
and let $w$ be the (up to normalization) unique positive corresponding eigenfunction.
Let then $Y(\theta):= 1-N \sin^2 \theta$, $\theta\in(-\pi/2,\pi/2)$ be the (up to sign and normalization) unique axially symmetric spherical harmonic of degree two.  If 
\begin{equation}\label{eq:alpha_1}
\alpha_1<-2N
\end{equation}
then $v=\mathfrak{v}(r,\theta)=w(r)Y(\theta)$ satisfies assumptions \eqref{radial-weakly-stable-assumption-1}--\eqref{radial-weakly-stable-assumption-4} in Proposition \ref{radial-weakly-stable}.
\end{lemma}
\begin{proof}
By construction, $v \in C^1_0(A)$ and satisfies assumptions \eqref{radial-weakly-stable-assumption-2} and \eqref{radial-weakly-stable-assumption-3} of Proposition~\ref{radial-weakly-stable}. Moreover, 
$$
\int_{\mathbb S^{N-1}} \mathfrak{v}(r,\cdot)\,d\sigma 
= w(r) \int_{\mathbb S^{N-1}}(1-N x_N^2)d\sigma(x) = w(r) \int_{\mathbb S^{N-1}}(1-|x|^2)d\sigma(x) = 0
$$
for every $r \in (R_0,R_1)$, so assumption \eqref{radial-weakly-stable-assumption-4} is also satisfied. It remains to prove \eqref{radial-weakly-stable-assumption-1}. To this aim, we recall that the function $Y$ is an eigenfunction of the Laplace-Beltrami operator $-\Delta_{\mathbb S^{N-1}}$ on the unit sphere $\mathbb S^{N-1}$ corresponding to the eigenvalue
$\lambda_2= 2N$. 
By using \eqref{eq:Laplace_radial}, it is straightforward to verify that 
\[
-\Delta v + v -(p-1)u_\mathrm{rad}^{p-2}v=\frac{\alpha_1+2N}{|x|^2} v \quad \mbox{in }A.
\]
By testing this equation by $v$ and integrating by parts, we obtain
\[
\int_A \left( |\nabla v|^2+v^2-(p-1)u_\mathrm{rad}^{p-2}v^2\right)\,dx =
(\alpha_1+2N) \int_A \frac{v^2}{|x|^2} \,dx <0,
\]
by assumption. Since the left hand side is $I''(u_\mathrm{rad})(v,v)$, the proof is concluded.
\end{proof}

\begin{proof}[Proof of Theorem~\ref{thm:main-a-const}]
By combining Proposition \ref{radial-weakly-stable} and Lemma \ref{lem:v_def}, it remains to prove the validity of relation \eqref{eq:alpha_1} under assumption \eqref{eq:suff_p}.

To this aim, notice that the eigenvalue $\alpha_1$ admits the variational characterization
    $$
    \alpha_1 = \min_{\varphi \in H^1_{0,\mathrm{rad}}(A) \setminus \{0\}} \frac{  \int_{A}\Bigl(|\nabla \varphi|^2+ \varphi^2\Bigr)\,dx - (p-1)\int_{A}u_{\mathrm{rad}}^{p-2}\varphi^2\,dx}{\int_{A}\frac{\varphi^2}{|x|^2}\,dx},
    $$
    where $H^1_{0,\mathrm{rad}}(A)$ denotes the subspace of radially symmetric functions in $H^1_0(A)$.
    Using in particular $\varphi = u_{\mathrm{rad}}$ as a test function, we obtain that 
\begin{align*}
  \alpha_1 &\le \frac{  \int_{A}\Bigl(|\nabla u_{\mathrm{rad}}|^2+ u_{\mathrm{rad}}^2\Bigr)\,dx - (p-1)\int_{A}u_{\mathrm{rad}}^{p}\,dx}{\int_{A}\frac{u_{\mathrm{rad}}^2}{|x|^2}\,dx}\\
  &=-(p-2)\frac{  \int_{A}\Bigl(|\nabla u_{\mathrm{rad}}|^2+ u_{\mathrm{rad}}^2\Bigr)\,dx}{\int_{A}\frac{u_{\mathrm{rad}}^2}{|x|^2}\,dx}.
\end{align*}
Since, 
$$
\int_{A}|\nabla u_{\mathrm{rad}}|^2\,dx \ge \Bigl(\frac{N-2}{2}\Bigr)^2  \int_{A}\frac{u_{\mathrm{rad}}^2}{|x|^2}\,dx
$$
by Hardy's inequality and
$$
\int_{A} u_{\mathrm{rad}}^2 \,dx > R_0^2 \int_{A}\frac{u_{\mathrm{rad}}^2}{|x|^2}\,dx,
$$
it follows that $\alpha_1 < -(p-2) \Bigl( \Bigl(\frac{N-2}{2}\Bigr)^2 +R_0^2\Bigr)$ and therefore $\alpha_1 < -2N$ by assumption \eqref{eq:suff_p}. Thus \eqref{eq:alpha_1} holds.
\end{proof}
    
\begin{remark}
It is proved in \cite[Proposition 4.5]{gladiali-et-al2010} that the first eigenvalue $\alpha_1$ of the eigenvalue problem~(\ref{eq:radial-eigenvalue-problem}) satisfies, as a function of the exponent $p>2$, the asymptotic expansion
\[
\alpha_1 = \alpha_1(p) = -c p^2 + o(p^2) \qquad \text{as $p \to \infty$ with a constant $c>0$}.
\]
This allows to conclude the weaker result that there exists $p_*>2$ such that $\alpha_1<-2N$ for every $p>p_*$.

Similarly,  as by \cite[Proposition 3.2]{gladiali-et-al2010} it holds
\[
\alpha_1 = \alpha_1(R) = -c R^2 + o(R^2) \qquad \text{as $R \to \infty$ with a constant $c>0$},
\]
being $A_R :=\{x\in\mathbb{R}^N\,:\, R<|x|<R+1\}$ an annulus with fixed width,  one obtains the existence of a nonradial solution on annuli with fixed width and sufficiently large radius $R$.
\end{remark}

\section*{Acknowledgments}
\noindent 
The authors acknowledge the support of the Departement of Mathematics of the University of Turin.
A. Boscaggin, F. Colasuonno and B. Noris were partially supported by the INdAM - GNAMPA Project 2019 ``Il modello di Born-Infeld per l'elettromagnetismo nonlineare: esistenza, regolarità e molteplicità di soluzioni'' and by the INdAM - GNAMPA Project 2020 ``Problemi ai limiti per l'equazione della curvatura media prescritta''. 
B. Noris acknowledges the support of the program S2R of the Université de Picardie Jules Verne, which financed a short visit to Turin, where part of this work has been achieved.
The authors also thank Susanna Terracini for helpful discussions.
\bibliographystyle{abbrv}
\def\cprime{$'$}

\end{document}